\documentclass[12pt, a4paper]{amsart}
\usepackage{amsfonts,amsmath, amsthm, amssymb, amscd}
\usepackage{latexsym}

\input{amssym.def}
\input{amssym}
\usepackage{hyperref}
\usepackage[dvips]{graphicx}
\usepackage{psfrag}

\newtheorem{defi}{Definition}[section]
\newtheorem{satz}[defi]{Theorem}
\newtheorem{prop}[defi]{Proposition}

\newtheorem{lemma}[defi]{Lemma}
\newtheorem{bem}[defi]{Remark}
\newtheorem{folg}[defi]{Corollary}

\newcommand {\N}{\mathbb{N}} 
\newcommand {\R}{\mathbb{R}} 
\newcommand {\C}{\mathbb{C}} 

\begin{document}
\allowdisplaybreaks
\pagestyle{plain}
\title{Dynamics of $L^p$-Multiplier for $p\leq 2$ on Harmonic Manifolds of purely exponential volume growth}

\author{Oliver Brammen}
\address{Faculty of Mathematics,
Ruhr University Bochum, 44780 Bochum, Germany}
\email{oliver.brammen@rub.de}
\thanks{The author would like to thank Kingshook Biswas for presenting the question on the behaviour of $L^p$-multipliers for $p\leq 2$ and   Gerhard Knieper and Norbert Peyerimhoff for helpful discussions and encouragement.  The author is partially supported by the German Research Foundation (DFG), CRC TRR 191, Symplectic structures in geometry, algebra and dynamics.}


\begin{abstract}
We study the dynamics of $L^p$-multipliers on non-compact simply connected harmonic manifolds of purely exponential volume growth. These are linear operators on the $L^p$-spaces which behave nicely on radial functions under Fourier transformation. 
In the process we complement results from \cite{MR4443685} by showing that if they are acting nicely on smooth function with compact support under Fourier transform for $p\leq 2$ these can not be chaotic. Furthermore, we use this to study the behaviour of the heat semi-group, the resolvent and the convolution algebra arising from convolution with radial functions. In the process, we obtain a Young inequality for the convolution on non-compact simply connected harmonic manifolds, an extension of the Kunz-Stein and study the domain of holomorphicity of the Fourier transform. 
\end{abstract}


\maketitle

\section{Introduction}
A linear operator on a separable Banach space is called chaotic if it has a dense set of periodic points and has a dense orbit, or equivalently, is topological transitive (also known as is hypercyclic).
The systematic study of the chaotic behaviour of semi-groups started in 1997 with an article by Desch, Schappacher and Webb \cite{desch_schappacher_webb_1997}. The study of the heat semi-group in this context was carried on in an article of Herzog \cite{herzog1997universality} and one by DeLaubenfels, Emamirad and Gross-Erdmann \cite{delaubenfels2003chaos}.
Ji and Weber \cite{JI20101121} showed that after shifting the heat semi-group on finite volume locally symmetric spaces of rank one it is chaotic on the space of radial $L^p$-functions for $p\in(1,2).$
Furthermore, the authors showed in \cite{ji_weber_2010} that for non-compact symmetric spaces of rank one, the same holds true in the H\"older conjugated range. This result was subsequently improved by Sakar \cite{Sarkar_2013} by showing that the shifted heat semi-group of a Demek Ricci space, which includes all non-compact symmetric spaces of rank one, is chaotic on the whole $L^p$ space for $p\in(2,\infty)$. Sakar and Pramanik \cite{PRAMANIK20142867} later showed that the same holds true for symmetric spaces of higher rank. Furthermore Ji and Weber \cite{ji_weber_2015} extended their previous result on finite volume locally symmetric spaces of rank one to higher rank. Finally,  Ray and Sakar  \cite{ray2018chaotic} extended the previous results on symmetric spaces of non-compact type to a more general class of operators called $L^p$-multiplier. Furthermore, they showed that these operators can not be chaotic in the range $p\in[1,2]$. This dichotomy is largely due to the lack of eigenfunctions of the Laplacian for $p\in[1,2]$. 
The main objective of this article is to extend the results of Ray and Sakar  to non-compact harmonic manifolds of purely exponential volume growth. That $L^p$-multipliers are chaotic for $p\in (2,\infty)$ was already shown by Biswas and Sarkar in \cite{MR4443685} for non-compact harmonic manifolds of purely exponential volume growth. 
We are going to complete the discussion of $L^p$-multiplier on those manifolds by showing that they can not be chaotic in the range $p\in [1,2]$ under the assumption that they act by multiplication with a holomorphic function under Fourier transform for $p\in[1,2)$ and with a continuous function for $p=2$. This behaviour is observed in all known examples. 
Furthermore, we will discuss examples of $L^p$-multiplier in Section 4 in the form of the heat semi-group and the resolvent. 
The second objective of this article is to provide some additional information on the convolution on non-compact simply connected harmonic manifolds of purely exponential volume growth defined in \cite{Biswas2019}. This will be done in the beginning of Section 3, since the convolution with a radial function is a prime example for an $ L^p$ multiplier.
We begin this article with a brief introduction to the Fourier transform on simply connected harmonic manifolds of purely exponential volume growth. For a detailed history of the study of harmonic manifolds, we refer to the surveys  \cite{kreyssig2010introduction} and \cite{knieper22016}.
\section{Preliminaries}
\subsection{Geometric Property of Harmonic Manifolds with Purely Exponential Volume Growth.}
In this section, we recall some facts needed in the proceedings. The main source of reference for those is \cite{Biswas2019}. A Riemannian manifold $(X,g)$ is called harmonic if for every point $\sigma\in X$ there exists a non-trivial radial, i.e. only on the geodesic distance $d(\sigma,x)$ dependent, solution of the Laplace equation 
$\Delta f=0$ on a punctured neighborhood of $\sigma$, where $\Delta=\operatorname{div}\operatorname{grad}$ is the Laplacian on $X$.  
An equivalent definition of $X$ being harmonic is that the density function is radial in geodesic coordinates, and therefore in this coordinate the density function $A(r)$ is given by  Jacobian of the map sending $v\in S_{\sigma}X$ to $\operatorname{exp}(rv)$. Using this definition and assuming that $X$ is additionally non-compact and simply connected, hence admits at least one minimal ray, it follows that $X$ have no conjugate points. Therefore by the Hadamard Cartan theorem the exponential map $\operatorname{exp}: T_{\sigma} X\to X$ is a diffeomorphism for every $\sigma\in X$.\\
The prime examples of harmonic manifolds are symmetric spaces of rank one and Euclidean spaces. 
It was a long-standing conjecture that all harmonic manifolds are of this type, referred to as the Lichnerowicz conjecture \cite{Lich}. The conjecture was proven for compact simply connected spaces by Szabo\cite{szabo1990} but shortly after in 1992  Damek and Ricci \cite{Damek_1992} provided for dimension 7 and higher a class of homogeneous non-compact simply connected harmonic spaces that are non-symmetric. These manifolds are called Damek-Ricci spaces. 
Furthermore the conjecture holds true for $\operatorname{dim}X< 6$  ( see \cite{Lich}, \cite {Walker}, \cite{Besse1978} and \cite{Nik}).
A non-compact simply connected harmonic manifold $(X,g)$ is said to be of purely exponential volume growth if there exists some constant $C\geq1$ and $\rho>0$ such that:
\begin{align*}
\frac{1}{C}\leq \frac{A(r)}{e^{2\rho r}}\leq C.
\end{align*}
This property is by \cite{knieper2009new} equivalent to
\begin{itemize}
\item Anosov Geodesic Flow
\item Gromov Hyperbolicity 
\item Rank one,
\end{itemize}
where the rank of a manifold without conjugate points is defined as the minimal dimension of the kernel of the symmetric operator given by the sum of the Hessian of the Busemann function in forward and backwards directions. By \cite{knieper2009new}  this definition coincides with the usual definition of rank in the setting of negative curvature (see \cite{Ballmann1985}  for the definition). 
All rank one symmetric spaces and  Damek-Ricci spaces are of this type, furthermore, Knieper in \cite{knieper2009new} also showed that non-positive curvature, or more generally no focal points, implies purely exponential volume growth. And that either of the conditions above implies the Lichnerowicz conjecture if $X$ admits a compact quotient.\\
From now on let $(X,g)$ be a non-compact simply connected harmonic manifold of rank one.
The geometric boundary $\partial X$ is defined by equivalence classes of geodesic rays. Where two rays are equivalent if their distance is bounded. 
The topology on $\partial X$ is the cone topology with the property that for $\overline{X}=X\cup\partial X$ and $B_1(x)=\{v\in T_xX\vert~\Vert v\Vert_g\leq 1\}$ the map $pr_x: B_1(x)\to \overline{X}$
\begin{align*}
pr_x (v)=\begin{cases}
\gamma_v(\infty)&\text{if}~\Vert v\Vert_g=1\\
\exp(\frac{1}{1-\Vert v \Vert}v)&\text{if}~\Vert v\Vert_g<1
\end{cases}
\end{align*}
is a homeomorphism.

 Since the geodesic flow is Anosov the Busemann function only depends on the direction of the ray. Hence 
for $x\in X$ and $\xi\in\partial X$ being the point at infinity of the  geodesic $\gamma$ we can alternatively define the  Busemann function $B_{\xi,x}:X\to\R$ by $$B_{\xi,x}(y)=\lim_{t\to\infty}(d(y,\gamma(t))-d(x,\gamma(t)).$$
Furthermore, Busemann functions are Lipschitz with Lipschitz constant $1$ and are analytic \cite{BusemannHarmonic}. Additionally, they have the properties: 
\begin{align*}
\lVert \operatorname{grad}B_{\xi,x}(y)\rVert_g=1,\quad \Delta B_{\xi,x}=2\rho,
\end{align*}
where $\rho$ is the constant above and $2\rho$ is the mean curvature of the horospheres. 

Additionally, the authors in \cite{Biswas2019} obtained a co-cycle property:

$$B_{\xi,x}=B_{\xi,\sigma}-B_{\xi,\sigma}(x)\quad \forall x,\sigma\in X.$$ 

Furthermore, by pushing forward the probability measure induced by the metric $\theta_x$ on $S_x X$ under $pr_x$  we obtain a  probability measure $\mu_x$ on $\partial X$. Hence, we have a family of probability measures $\{\mu_x\}_{x\in X}$, that are pairwise absolutely continuous with Radon-Nikodym derivative  $\frac{d\mu_x}{d\mu_y}(\xi)=e^{-2\rho B_{\xi,x}(y)}$. For a detailed proof see \cite{Knieper2016}.

Let $f$ be a $C^2$ function on $X$  and $u$ a $C^{\infty}$ function on $\R$. Then we have:
$$\Delta (u\circ f)=(u''\circ f)\lVert \operatorname{grad} f\rVert_g^2 +(u'\circ f)\Delta f.$$
(See for instance \cite[Lemma 3.1]{Biswas2019}.)

With this, we can calculate the spherical and horospherical part of the Laplacian, by
choosing $f=d_x$ for some $x\in X$. We obtain with $\Delta d_x(r)=\frac{A^{\prime}(r)}{A(r)}\circ d_x(r)$ using spherical coordinates around $x$
\begin{align}\label{equ:radiallaplace}
\Delta(u\circ d_x)=u^{\prime\prime}\circ d_x +u^{\prime}\circ d_x\cdot \frac{A^{\prime}}{A}\circ d_x.
\end{align}
For the Busemann function  $f=B_{\xi,x}$ with $\Delta b_v=h$ we obtain using horospherical coordinates 
\begin{align}
\Delta( u\circ B_{\xi,x})=u^{\prime\prime}\circ B_{\xi,x}+2\rho\cdot u^{\prime}\circ B_{\xi,x}.
\end{align}
Therefore we have:
 $g(y)=e^{(i\lambda-\rho)B_{\xi, x}(y)}$ is a eigenfunction of the Laplacian with $g(x)=1$ and $\Delta g= -(\lambda^2+\rho^2)$ for $\lambda\in \C$.
 

\subsection{Fourier Transform and Plancherel Theorem on Rank One Harmonic Manifolds}
The main tool in defining the Fourier transform on  harmonic manifolds of purely exponential volume growth is the theory of hypergroups or more generally the study of operator given by the radial part of the Laplacian. This was first presented for harmonic manifolds with pinched negative curvature in \cite{biswas2018fourier} and then extended in \cite{Biswas2019} to rank one harmonic manifold.

Since we refrain from details, we refer the reader 
 to \cite{bloom2011harmonic} for a thorough discussion of the topic and the definition. In \cite{Biswas2019} the authors showed that the density function $A(r)$ of a harmonic manifold of rank one  satisfies the following conditions   %
\begin{itemize}
\item[(C1)] $A$ is increasing and $A(r)\to\infty$ for $r\to \infty$.
\item[(C2)] $\frac{A^\prime}{A}$ is decreasing and $\rho=\frac{1}{2} \lim\limits_{r\to\infty} \frac{A^\prime(r)}{A(r)}>0$.
\item[(C3)] For $r>0$, $A(r)=r^{2\alpha +1 }B(r)$ for some $\alpha>-\frac{1}{2}$ and some even $C^{\infty}$ function $B(x)$ on $\R$ with $B(0)=1$.
\item[(C4)] 
\begin{align*}
G(x)=\frac{1}{4}\Bigl(\frac{A^\prime}{A}(r)\Bigr)^2 +\frac{1}{2}\Bigl(\frac{A^\prime}{A}(r)\Bigr)^\prime -\rho^2
\end{align*}\label{C42}
is bounded on $[r_0,\infty)$ for all $r_0>0$  and 
    \begin{align*}
\int_{r_1}^\infty r\vert G(r)\vert \,dr<\infty\quad\text{for some}~r_1 >0.
\end{align*}
\end{itemize}
Therefore $A(r)$ defines a   Ch\'ebli-Trim\'eche hypergoup. The structure of the so-defined hypergroup is related to the second-order differential operator given by the radial part of the Laplacian: 
\begin{align}\label{eq:A}
 L_{A}:=\frac{d^2}{dr^2}+\frac{A^{\prime}(r)}{A(r)}\frac{d}{dr}.
\end{align}
Let
\begin{align}
\varphi_{\lambda}:\R^+\to\R, \quad\lambda\in [0,\infty)\cup [0,i\rho]
\end{align}
be the eigenfunction of $L_A$ with
\begin{align}
L_{A}\varphi_{\lambda}=-(\lambda^2+\rho^2)\varphi_\lambda
\end{align}
and which admit a smooth extension to zero with $\varphi_{\lambda}(0)=1$.
 Under  conditions C1-C4 it was shown in \cite{Bloom1995TheHM}  that there is a complex function $\mathbf{c}$ on $\C\setminus \{0\}$. Such that 
for the two linear independent solutions to $L_{A,x}u=-(\lambda^2+\rho^2)$ 
$\Phi_{\lambda}$ and $\Phi_{-\lambda}$ which are asymptotic to exponential functions, meaning:
\begin{align}
\Phi_{\pm \lambda}(r)= e^{(\pm i\lambda-\rho)r}(1+o(1))\text{ as }r\to\infty,
\end{align}
we have
\begin{align}\label{eigende}
\varphi_{\lambda}=\mathbf{c}(\lambda)\Phi_{\lambda}+\mathbf{c}(-\lambda)\Phi_{-\lambda}\quad\forall \lambda\in \C\setminus \{0\}.
\end{align}
The authors in \cite{Biswas2019} then defined the radial Fourier transform by: 
\begin{defi}
Let $f:X\to\R$ be, i.e. $f=u\circ d_x$ for some $x\in X$, where $u:[0,\infty)\to\R$ and $d_x:X\to \R$ is the distance function. The Fourier transform of $f$ is given by: 
$$\widehat{f}(\lambda):=\omega_{n-1}\int_0^{\infty}u(r)\varphi_{\lambda}(r)A(r)\,dr,$$
\end{defi}
where $\omega_{n-1}$ is the volume of $S^{n-1}\subset\R^n$. 
Note that in the following we will omit to mention the base point $x$ unless there is the possibility of confusion. By $f$ radial, we mean that  $f$ is radial around some point $\sigma\in X$. Further, we will use this point as a base point for the radial Fourier transform unless stated otherwise. 
Now observe that we obtain the radial eigenfunctions of the Laplace operator with eigenvalue $-(\lambda^2+\rho^2)$ by:
\begin{align}\label{defi:ef}
\varphi_{\lambda,x}(y)=\varphi_{\lambda}\circ d(x,y)\quad \forall x,y\in X.
\end{align}
With \cite{Bloom1995TheHM}  the authors in \cite{Biswas2019} showed that there is a constant $C_0$ such that 
for $f\in L^1(X)$ radial, i.e. $f=u\circ d_x$ for some $x\in X$ and $u:[0,\infty)\to\R$ such that $\widehat{u}\in L^1((0,\infty),C_0\lvert \mathbf{c}(\lambda)\rvert^{-2}\,d\lambda)$.
\begin{align}\label{fourad}
f(y)=C_0\int_{0}^\infty \widehat{f}(\lambda)\varphi_{\lambda,x}(y)\vert \mathbf{c}(\lambda) \vert^{-2}\, d\lambda.
\end{align}
And the radial Fourier transform extends to an isometry between the radial $L^2$ functions denoted by $L_\sigma^2(X)$ and $$L^2((0,\infty),C_0\lvert \mathbf{c}(\lambda)\rvert^{-2}\,d\lambda).$$
\begin{bem}
More precise the authors in \cite{Bloom1995TheHM} showed under the additional condition that $\lvert \alpha \rvert>\frac{1}{2}$ that $\mathbf{c}$ dose does not have zeros on the closed lower half plane. Hence this would exclude the case $\operatorname{dim} X=3$ (see \cite{Biswas2019}) but the Lichnerowicz conjecture is affirmed in $\operatorname{dim} X<6$ (see  \cite{Lich}, \cite {Walker}, \cite{Besse1978}, \cite{Nik}) and therefore the Jacobin analysis applies, and we can use the $\mathbf{c}$-function obtained in this context.
\end{bem}
\begin{bem}
Using the work of Chebli and the definition of the Fourier transform on harmonic manifolds it follows that:
$$C_{0}=\frac{1}{2\pi \omega_{n-1}}.$$
This clarification is due to a talk given by Gerhard Knieper at the International Conference on 
Einstein and Harmonic Manifolds on the 21st of June 2023. 
\end{bem}

 In the same fashion as in the case of the Halgason Fourier transform on symmetric spaces the authors in \cite{Biswas2019} extended the Fourier transform to non-radial functions. Namely, by using the radial symmetry of the Poisson kernel.
\begin{defi}
		Let $\sigma\in X$ for $f:X\to \C$ measurable, the Fourier transform of $f$ based at $\sigma$ is given by 		
		$$\tilde{f}^{\sigma}(\lambda,\xi)=\int_{X}f(y)e^{(-i\lambda-\rho)B_{\xi,\sigma}(y)}\,dy$$
		for $\lambda\in\C$, $\xi\in\partial X$ for which the integral above converges and  $B_{\xi,\sigma}$ the Busemann function in direction  $\xi$ based at $\sigma$, i.e.
			$$B_{\xi,\sigma}(\sigma)=0.$$
\end{defi}
		We can immediately note that because of the co-cycle property the Fourier transforms with respect to different base points $x,\sigma\in X$ are related by: 
		
		\begin{align}\label{eq:fourpoint}
			\tilde{f}^x (\lambda,\xi)=e^{(i\lambda+\rho)B_{\xi,\sigma}(x)}\tilde{f}^{\sigma}(\lambda,\xi).
			\end{align}
	This can be seen by the following calculation:
			Let $x,\sigma\in X$ and $f\in C^{\infty}_c(X)$ then we have for $\lambda\in \C$ and $\xi\in \partial X$ that:
		\begin{align*}
		\tilde{f}^x (\lambda,\xi)&=\int_Xf(y)e^{(-i\lambda-\rho)B_{\xi,x}(y)}\,dy\\
		&=\int_Xf(y)e^{(-i\lambda-\rho)B_{\xi,\sigma}(y)}\cdot e^{(i\lambda+\rho)B_{\xi,\sigma}(x)}\,dy\\
		&=e^{(i\lambda+\rho)B_{\xi,\sigma}(x)}\int_Xf(y)e^{(-i\lambda-\rho)B_{\xi,\sigma}(y)}\,dy\\
		&=e^{(i\lambda+\rho)B_{\xi,\sigma}(x)}\tilde{f}^{\sigma}(\lambda,\xi).
		\end{align*}
		Furthermore, the Fourier transform coincides with the radial Fourier transform on radial functions. For details see \cite[Lemma 5.2]{Biswas2019}. 	
	The inversion formula follows now from the representation of the radial eigenfunctions via a convex combination of non-radial eigenfunctions, \cite[Theorem 5.6]{Biswas2019},:
			\begin{align}\label{radialeigen}
			\varphi_{\lambda,x}(y):=\int_{\partial X}e^{(i\lambda-\rho)B_{\xi,x}(y)}\,d\mu_x (\xi).
		\end{align}
This is analogous to the well-known formula on a rank one symmetric space $G/K$ and harmonic $NA$ groups see for the symmetric case \cite[Chapter III, Section 11]{helgason1994geometric} and for the harmonic $NA$ group \cite{Damek1992} and \cite{fourierNA2}.
		Using this the authors in \cite{Biswas2019} obtained:
			\begin{align*}
			f(x)=C_0 \int_{0}^{\infty}\int_{\partial X}\tilde{f}^{\sigma}(\lambda,\xi)e^{(i\lambda-\rho)B_{\xi,\sigma}(x)}\,d\mu_{\sigma}(\xi)
			\vert \mathbf{c}(\lambda)\vert^{-2}\,d\lambda
			\end{align*}
		where $C_0$ is the same constant given in (\ref{fourad}).
		Additionally, they also obtained a Plancherel theorem:
\begin{satz}[\cite{Biswas2019}]\label{Plancherel Theorem}
		Let $\sigma\in X$, $f,g\in C_{c}^{\infty}(X)$. Then we have:
			$$\int_X f(x)\overline{g(x)}\,dx=C_{0}\int_{0}^{\infty}\int_{\partial X}\tilde{f}^{\sigma}(\lambda,\xi)\overline{\tilde{g}
			^{\sigma}(\lambda,\xi)} \vert \mathbf{c}(\lambda)\vert^{-2}\,d\mu_{\sigma}(\xi)d\lambda 
			$$ 
		and the Fourier transform extends to an isometry between $L^2(X)$ and\\
		 $L^2((0,\infty)\times \partial X,C_0 \vert \mathbf{c}(\lambda)\vert^{-2} \,d\mu_{\sigma}(\xi)\,d\lambda).$
\end{satz}

\section{Main Results}
\subsection{The Convolution}\label{Covolution}
Assume $(X,g)$ to be a non-compact simply connected harmonic manifold of rank one. 
 Let $f$ be radial around $\sigma\in X$ with $f=u\circ d_{\sigma}$ for some function $u:\R\to \R$. For $x\in X$ define the $x$-translate of $f$ by 
\begin{align}
\tau_x f :=u\circ d_x.
\end{align}
Observe that 

\begin{enumerate}
\item $\tau_x f(y)=u\circ d(x,y)=\tau_{y}f(x)$.
\item If $f\in L^1 (X)$ is radial  we have $\lVert f \rVert_1=\omega_{n-1}\int_{0}^{\infty} \lvert u(r)\rvert A(r)\,dr=\lVert \tau_x f \rVert_1$.
\item  If $f\in L^2 (X)$ is radial  we have $\lVert f \rVert_2=\omega_{n-1}\int_{0}^{\infty} \lvert u(r)\rvert^2 A(r)\,dr=\lVert \tau_x f \rVert_2$.
\end{enumerate}
\begin{defi}
Let $g:X\to \C$ be a radial function  around $\sigma\in X$, with $g=u\circ d_\sigma$ for some function $u:\R_{\geq 0}\to \C$. 
 Then the convolution of a function $f:X\to \C$ with compact support with $g$
   is given by $$ (f*g)(x):=\int_{X} f(y)\cdot\tau_x g(y)\,dy.$$
\end{defi}
Now we need to check that the integral above is defined for almost every $x\in X$, if $g$ and $f$ are in $L^1(X)$.  For that purpose let $g=u\circ d_{\sigma}$. Then we have:
\begin{align*}
\lVert f*g\rVert_1&\leq \int_X\int_X\lvert f(y)\rvert\lvert(\tau_xg)(y)\rvert \,dy\,dx\\
&=\omega_{n-1}\int_X\lvert f(y)\rvert\Bigl(\int_{0}^{\infty}\lvert u(r)\rvert A(r) \,dr\Bigr) \,dy\\
&=\lVert f\rVert_1\lVert g\rVert_1< \infty. 
\end{align*}
Moreover, by replacing $f\in L^1(X)$ by $f\in L^{\infty}(X)$ in the calculation above we obtain:
\begin{align*}
\lVert f*g\rVert_{\infty}\leq \lVert f\rVert_{\infty}\cdot \lVert g\rVert_1.
\end{align*}
Hence, by the Riesz-Thorin interpolation theorem (see for instance \cite[Chapter 4]{bennett1988interpolation}), we get for 

 for all $ p\in[1,\infty]$ and $f\in L^p(X),$
\begin{align}\label{eq:conalg}
\lVert f*g\rVert_p\leq \lVert f\rVert_{p}\cdot\lVert g\rVert_1.
\end{align}

\begin{lemma}[\cite{MR4443685}]\label{id}
Let $f\in L^p(X)$ and $\psi_n$ be a radial approximation of the identity such that $\psi_n\geq 0$, $\int_X\psi_n (x) \,\,dx=1$
and
$$\lim_{n\to\infty}\int_{B(\sigma,r)}\psi_n (x) \,dx=1\quad \forall r>0.$$
Then $\lVert f*\psi_n-f\rVert_p\to 0,\text{ for } n\to \infty.$
\end{lemma}

\begin{lemma}[\cite{Biswas2019}]\label{f+g}
Let $f:X\to\C$ be smooth with compact support and $g:X\to\C$ smooth with compact support and radial around $\sigma\in X$. Then we have:
$$\widetilde{f*g}^{\sigma}=\tilde{f}^\sigma\cdot \widehat{g}.$$
\end{lemma}
\begin{bem}\label{Bem2.4.5}
From the above we also obtain that if for $f,g:X\to\C$ the Fourier transform exists for almost every $\xi\in \partial X$ and $\lambda\in \R$ and if $g$ is radial then we have 
$$\widetilde{f*g}^{\sigma}=\tilde{f}^\sigma\cdot \hat{g}$$
almost everywhere. 
\end{bem}
\begin{folg}\label{folg:conv}
By the Plancherel Theorem for the Fourier transform we get
$$ \lVert f*g\rVert_2 =\lVert \tilde{f}^\sigma\cdot \hat{g}\rVert_2.$$
for suitable functions $f,g$. 
\end{folg}

\begin{lemma}\label{lemma:convextention}
Let $g\in L_{\sigma}^2(X)$ be radial. Then the map $f\mapsto f*g$ from $C^{\infty}_c(X)$ to $L^2(X)$ admits a continuous extension to a map from $L^2(X)$ to $L^2(X)$ if and only if $\lVert \widehat{g}^{\sigma}\rVert_{\infty}< \infty$. 
\end{lemma}
\begin{proof}
Lemma \ref{f+g} and the density of smooth radial function with compact support in $L^2_{\sigma}(X)$ imply:
\begin{align*}
\widetilde{f*g}^{\sigma}=\tilde{f}^\sigma\cdot \widehat{g}^{\sigma} \quad \forall f\in C^{\infty}_c(X)
\end{align*}
and the Plancherel theorem yields:
\begin{align*}
\lVert f*g\rVert^2_2=C_{0}\int_{0}^{\infty}\int_{\partial X}\lvert \tilde{f}^{\sigma}(\lambda,\xi)\rvert^2 \lvert \widehat{g}
			^{\sigma}(\lambda) \rvert^2 \vert \mathbf{c}(\lambda)\vert^{-2}\,d\mu_{\sigma}(\xi)\,d\lambda.
\end{align*}
Suppose now that the map  $f\mapsto f*g$ extends to a continuous  map  from $L^2(X)$ to $L^2(X)$ also denoted by $f*g$ then we claim that  
\begin{align}\label{eq:conl2}
\widetilde{f*g}^{\sigma}=\tilde{f}^\sigma\cdot \widehat{g}^{\sigma} \quad \forall f\in L^2(X),
\end{align}
almost everywhere.
If $f\in L^2(X)$ there is a sequence $f_n\in C^{\infty}_c(X)$ converging to $f$ in $L^2(X)$. Hence $f_n*g\to f*g$ in $L^2(X)$ and by the Plancherel theorem
\begin{align*} \widetilde{f_n*g}\to \widetilde{f*g}
\end{align*}
in $L^2([0,\infty)\times \partial X,C_0 \vert \mathbf{c}(\lambda)\vert^{-2}\,d\mu_{\sigma}(\xi)\,d\lambda)$. Additionally we have $\widetilde{f_n*g}^{\sigma}=\widetilde{f_n}^{\sigma}\cdot \widehat{g}^{\sigma}$, and 
\begin{align*}
\tilde{f_n}^{\sigma}\to \tilde{f}^{\sigma}
\end{align*}
in $L^2([0,\infty)\times \partial X,C_0 \vert \mathbf{c}(\lambda)\vert^{-2}\,d\mu_{\sigma}(\xi)\,d\lambda)$.
After if necessary passing on to a sub-sequence we can assume $\tilde{f_n}^{\sigma}\to \tilde{f}^{\sigma}$  almost everywhere therefore
$\widetilde{f_n}^{\sigma}\cdot \widehat{g}^{\sigma}\to \widetilde{f}^{\sigma}\cdot \widehat{g}^{\sigma}$ almost everywhere. 
Now the operator of multiplication with $\widehat{g}^{\sigma}$ is bounded on $L^2((0,\infty)\times \partial X,C_0 \vert \mathbf{c}(\lambda)\vert^{-2}\,d\mu_{\sigma}(\xi)\,d\lambda)$ if and only if $\lVert \widehat{g}^{\sigma}\rVert_{\infty}<\infty$.\\
 
Conversely if $\lVert \widehat{g}^{\sigma}\rVert_{\infty}<\infty$ then by the Plancherel theorem 
\begin{align*}
\lVert f*g\rVert_2\leq \lVert \widehat{g}^{\sigma}\rVert_{\infty}\cdot\lVert f\rVert_2.
\end{align*}
Hence the operator $f\mapsto f*g$ extends to a continuous operator from $L^2(X)$ to $L^2(X)$ and equation (\ref{eq:conl2}) remains valid analogous to the considerations above. 
\end{proof}

Using the Fourier inversion formula, the inequality (\ref{eq:conalg}) and the density of $C^{\infty}_c(X)$ in $L^{1}(X)$ we can deduce that for $f,g\in L^1(X)$ radial the convolution $ f*g=g*f\in L^1(X)$ is radial around the same point. Hence, the radial functions $L^1$-function, denoted by $L_{\sigma}^1(X)$, form a commutative Banach algebra under convolution \cite[Theorem 7.2]{Biswas2019}. Note that this also can be shown without the assumption of rank one as done in \cite{PS15}.\\

In light of this, we observe
that the measure algebra $M(\R^+,A(r))$ of the hypergroup can be realised by radial measures on $X$ by defining the convolution of radial measure via approximation by radial functions. Hence, $L^{1}_{\sigma}(X)$ forms a commutative sub-algebra of $M(\R^+A(r))$.
To finish this short introduction to the convolution with radial functions let us briefly state the Kunz-Stein phenomenon proven in \cite[Theorem 8.3]{Biswas2019}, show an extension of it for the convolution with radial eigenfunctions and  prove an inequality of Young-type, both of which will be important in the subsequent discussions. 
\begin{satz}[\cite{Biswas2019}]\label{kunzestein}
Let $(X,g)$ be a simply connected harmonic manifold of rank one. 
Let $x\in X$ and $1\leq p<2$. Let $g\in C_c^{\infty}(X)$ be radial around $x$. Then for any $f\in C_c^{\infty}(X)$ we have 
$$ \lVert f*g\rVert_2\leq C_p\lVert g\rVert_p\lVert f\rVert_2$$
for some constant $C_p>0$ which only depends on $p$. Moreover for any $g\in L^p(X)$ radial around $x$ the map $f\to f*g$ extends to a bounded linear operator on $L^2(X)$ with operator norm at most $C_p\lVert g\rVert_p$.
\end{satz}

We can extend this result using interpolation arguments to the convolution with radial eigenfunctions:
\begin{satz}
Let $(X,g)$ be a simply connected harmonic manifold of rank one and $\sigma\in X$.
If $f\in C^{\infty}_c(X)$, $q\in [1,2)$ and $p$ Hölder conjugated to $q$ then:
\begin{align*}
\lVert f*\varphi_{\lambda,\sigma}\rVert_p \leq C_q \lVert f \rVert_q\quad \forall \lambda\in \R
\end{align*}
for some constant $C_q>0$ only depending on $q$, with $C_1=1$.
\end{satz}
\begin{proof}
First note that 
\begin{align}\label{eq:unglvarphi}
\lvert \varphi_{\lambda,\sigma}(x)\rvert \leq \varphi_{0,\sigma}(x)\leq 1 \quad\forall \lambda\in\R
\end{align}
 for all $x\in X$. Now for $p=1$ we have by the Hölder inequality
 \begin{align*}
 f*\varphi_{\lambda,\sigma}(x)&=\int_X f(y)\varphi_{\lambda,x}(y)\,dy\\
 &\leq \lVert f \rVert_1\cdot \lVert \varphi_{\lambda,\sigma}\rVert_{\infty}\\
 &\leq  \lVert f \rVert_1.
 \end{align*}
 Hence $\lVert  f*\varphi_{\lambda,\sigma} \rVert_{\infty}\leq \lVert f \rVert_1$.
 Observe now that by (\ref{eq:unglvarphi}) it is sufficient to consider convolution with $\varphi_{0,\sigma}$. 
 Let $q\in (1,2)$, $y\in\R$ and $\epsilon=\frac{2-q}{4(q-1)}>0$. For $\theta\in [0,1]$ and $f\in C_c^{\infty}(X)$ define the family of linear operators $T_{\theta+iy}$ by
 \begin{align*}
 T_{\theta+iy}f=f*\varphi_{0,\sigma}^{1+(\theta-1)/2 +\epsilon\theta+iy}.
 \end{align*}
 Then for $\theta =0$ we have
 \begin{align*}
 \lVert f*\varphi_{0,\sigma}^{1/2+iy}\rVert_{\infty}&\leq \lVert \lvert f \rvert *\varphi_{0,\sigma}^{1/2}\rVert_{\infty}\\
 &\leq\lVert f\rVert_1.
 \end{align*}
 We now observe that by Remark \ref{Bem2.4.5} :
 \begin{align*}
 \widetilde{f*\varphi_{0,\sigma}^{1+\epsilon}}^{\sigma}=\tilde{f}^{\sigma}\cdot \widehat{\varphi_{0,\sigma}^{1+\epsilon}}^{\sigma}
 \end{align*}
 almost everywhere and 
 \begin{align*}
 \lvert \widehat{\varphi_{0,\sigma}^{1+\epsilon}}^{\sigma}(\lambda)\rvert &=\Bigl\lvert \int_X  \varphi_{0,\sigma}^{1+\epsilon}(x)\varphi_{\lambda,\sigma}(x)\,dx\Bigr\rvert\\
 &\leq \int_X  \varphi_{0,\sigma}^{2+\epsilon}(x)\,dx\\
 &< \infty.
 \end{align*}
 Therefore we can conclude using the Plancherel theorem:
 \begin{align*}
 \lVert T_{1+iy}f\rVert_2&\leq \lVert \lvert f\rvert* \varphi_{0,\sigma}^{1+\epsilon}\rVert_2\\
 &=\lVert  \widetilde{\lvert f\lvert *\varphi_{0,\sigma}^{1+\epsilon}}^{\sigma}\lVert_{L^2(0,\infty)\times X,C_0 \,d\mu_{\sigma} \lvert \mathbf{c}(\lambda) \rvert^{-2}\,d\lambda)}\\
 &\leq C_{\epsilon}\lVert f\rVert_2. 
 \end{align*}
 Where $C_{\epsilon}$ only depends on $\epsilon$.
 Therefore  choosing $\theta$ such that $1-\theta/2=1/q$ we get $(\theta-1)/2+\epsilon \theta=0$. And Steins's interpolation theorem (See for intance \cite[Chapter4]{bennett1988interpolation} )yields 
 \begin{align*}
 \lVert f*\varphi_{0,\sigma}\rVert_p \leq C_q \lVert f \rVert_q\quad \forall \lambda\in \R.
 \end{align*}
 And therefore by (\ref{eq:unglvarphi}) we have:
 \begin{align*}
\lVert f*\varphi_{\lambda,\sigma}\rVert_p \leq C_q \lVert f \rVert_q\quad \forall \lambda\in \R.
\end{align*}
 
\end{proof}

\begin{lemma}\label{young}
Let $p,q,r>0$ and satisfy: $1+\frac{1}{r}=\frac{1}{p}+\frac{1}{q}$. Then for $f\in L^p (X)$ and $g\in L^q(X)$, $g=u\circ d_\sigma$ radial around $\sigma\in X$ we get:
\begin{align*}
&f*g\in L^r(X)\\
&\text{ and }\\
&\Vert f*g\Vert_{r}\leq \Vert f\Vert_p \cdot \Vert g \Vert_q.
\end{align*}

\end{lemma}
\begin{proof}
First:
\begin{align*}
   & \vert (f*g) (x)\vert\leq \int_{X}\vert f(y)\vert \vert \tau_x g(y)\vert \,dy\\
&=\int_X \vert f(y) \vert^{1+p/r -p/r} \cdot \vert \tau_x g(y)\vert^{1+q/r-q/r} \,dy\\
&=\int_X \vert f(y)\vert^{p/r}\cdot \vert \tau_x g(y)\vert^{q/r}\cdot \vert f(y)\vert^{1-p/r}\vert \tau_x g(y)\vert^{1-q/r}\,dy\\
&=\int_X (\vert f(y)\vert^{p}\cdot \vert \tau_x g(y)\vert^{q})^{1/r}\vert f(y)\vert^{(r-p)/r}\vert \tau_x g(y)\vert^{(r-q)/r} \,dy.
\end{align*}
Since \begin{align*}
    \frac{1}{r}+\frac{r-p}{pr}+\frac{r-q}{rq}&=\frac{1}{r}+\frac{1}{p}-\frac{1}{r}+\frac{1}{q}-\frac{1}{r}\\&=\frac{1}{p}+\frac{1}{q}-\frac{1}{r}=1,
\end{align*}
we can use the generalised H\"older inequality to obtain:
\begin{align*}
 \vert (f*g) (x)\vert&\leq \Vert (\vert f(y)\vert^{p}\cdot \vert \tau_x g(y) \vert^{q})^{1/r}\Vert_r\\
 & \cdot \Vert \vert f(y)\vert^{(r-p)/r}\Vert_{\frac{pr}{r-p}}\cdot \Vert \vert\tau_x g(y)\vert^{(r-q)/r}\Vert_{\frac{qr}{r-q}}.
\end{align*}
 We now look at these three expressions separately
\begin{align*}
&\Vert (\vert f(y)\vert^{p}\cdot \vert \tau_x g(y) \vert^{q})^{1/r}\Vert_{r}\\
&=\Bigl(\int_X (\vert f(y)\vert^{p}\cdot \vert \tau_x g(y)\vert^{q})^{(1/r)\cdot r} \,dy\Bigr)^{1/r}\\
&=\Bigl(\int_X \vert f(y) \vert^{p}\cdot\vert \tau_x g(y)\vert^{q}\,dy\Bigr)^{1/r}\\
\Vert \vert f(y)\vert^{(r-p)/r}\Vert_{\frac{pr}{r-p}}&=\Bigl(\int_X \vert f(y)\vert^{\frac{r-p}{r}\cdot\frac{qr}{r-p}}\,dy\Bigr)^{\frac{r-p}{pr}}\\
&=\Bigl(\int_X \vert f(y)\vert^{p}\,dy\Bigr)^{\frac{1}{p}\cdot \frac{r-p}{r}}\\
&=\Vert f \Vert_{p}^{(r-p)/r}\\
\Vert \vert\tau_x g(y)\vert^{(r-q)/r}\Vert_{\frac{qr}{r-q}}
&=\Bigl(\int_X \vert \tau_x g(y)\vert^{\frac{r-q}{r}\cdot \frac{qr}{r-q}}\,dy\Bigr)^{\frac{r-q}{qr}}\\
&=\Bigl(\int_X \vert \tau_x g(y)\vert^{q}\,dy\Bigr)^{\frac{1}{q}\cdot\frac{r-q}{r}}\\
&=\Vert g\Vert_{q}^{(r-q)/r}.
\end{align*}

Now we conclude
\begin{align*}
&\Vert(f*g)(x)\Vert_{r}^{r}
=\int_X \vert (f*g)(x)\vert^{r}\,dx\\
\leq& \int_X \Bigl(\int_X(\vert f(y)\vert^{p}\cdot \vert \tau_x g(y)\vert^{q})^{1/r}\cdot \vert f(y)\vert^{(r-p)/r}\cdot \vert \tau_x g(y)\vert^{(r-q)/r}\,dy\Bigr)^{r}\\
&=\int_X \Bigl(\int_X (\vert f(y)\vert^{p}\cdot \vert \tau_x g(y)\vert^{q})\,dy\Bigr)\Vert f\Vert_{p}^{r-p}\Vert g\Vert^{r-q}\,dx\\
&=\Vert f\Vert_{p}^{r-p}\Vert g\Vert^{r-q}\int_X\int_X \vert \tau_x g(y)\vert^{q}\vert f(y)\vert^{p}\,dy\,dx\\
&=\Vert f\Vert_{p}^{r-p}\Vert g\Vert^{r-q}\int_X\vert f(y)\vert^{p}\,dy\int_X \vert \tau_x g(y)\vert^{q}\,dx\\
&=\Vert f\Vert_{p}^{r-p}\Vert g\Vert^{r-q}\int_X\vert f(y)\vert^{p}\,dy\omega_{n-1}\int_{0}^{\infty}\lvert u(r)\rvert^q A(r)\,dr\\
&=\Vert f\Vert_{p}^{r-p}\Vert g\Vert^{r-q}\Vert f\Vert_{p}^{p}\cdot\Vert g\Vert_{q}^{q}\\
&=\Vert f\Vert_{p}^{r}\cdot\Vert g\Vert_{q}^{r}.
\end{align*}
Hence, the assertion follows by taking the $r$-root. 
\end{proof}
\begin{bem}
As apparent from the proof the  Young-type inequality holds without the assumption of rank one. 
\end{bem}

\subsection{$L^p$-multiplier}

Let $q\in[0,2)$ and let $f,g\in C^{\infty}_c(X)$ with $g$ radial. By the discussion from the previous section we have
\begin{align*}
 \lVert f*g\rVert_2\leq C_p\lVert f\lVert_2\cdot\lVert g\rVert_q
 \end{align*} 
 for $C_q$ 
as in Theorem \ref{kunzestein}
  and
  \begin{align*}
  \lVert f*g\rVert_{\infty}\leq \lVert f\rVert_{\infty}\cdot\lVert g\rVert_1.
  \end{align*} 

By  the Riesze-Thorin theorem  
 we conclude
\begin{align*}
\lVert f*g\rVert_p\leq C_p \lVert f\rVert_p\cdot\lVert g\rVert_r
\end{align*}
 for $p\geq 2,~r\in [1,2)$ such that $\frac{1}{r}\leq 1+\frac{1}{p}$ and $C_p$  a positive constant. This implies that by convolution with a radial $L^r$-
 function we obtain a bounded linear operator 

\begin{align*}
T:L^p(X)&\to L^p(X)\\
f&\mapsto f*g
\end{align*}
such that 
\begin{enumerate}
\item $T$ preserves the subspace of radial functions,
\item for all radial functions $\psi,\phi$ we have $$T\psi*\phi=\psi*T\phi,$$
\item for all radial $C_c^{\infty}(X)$ functions $\phi$ and $x\in X$ we have $$T\tau_x \phi=\tau_xT\phi.$$
\end{enumerate}
This leads to the next definition:
\begin{defi}
For $p\in [1,\infty]$ an $L^p$-multiplier is a bounded linear operator $T:L^p(X)\to L^p(X)$ such that conditions 1-3 above are satisfied. 
\end{defi}
\begin{bem}
By the considerations above the convolution with a radial $L^1$ function defines a $L^p$-multiplier for all $p\in [1,\infty]$. Another example is the convolution with radial measures, see for instance \cite{PS15}, \cite{biswas2018dynamics} and \cite{MR4443685}.
\end{bem}

The next lemma is of vital importance not just for the discussion of the dynamics of an $L^p$-multiplier but also in the proof of the Kurz-Stein phenomenon which is due to \cite[Lemma 8.1]{Biswas2019}. 
\begin{lemma}[\cite{Biswas2019}] \label{lemma2.4.9}
Let $x\in X$, $q>2$ and $\gamma_q=1-\frac{2}{q}$. Then for $t\in(-\gamma_q\rho,\gamma_q\rho)$ and any $\lambda\in \C$ with $\operatorname{Im}\lambda=t$ we have
$$\lVert \varphi_{\lambda,x}\rVert_q\leq \lVert \varphi_{it,x}\rVert_q< \infty.$$
In particular $\varphi_{\lambda,x}\in L^q(X)$ for $\lambda\in S_q:=\{\lambda\in \C\mid \lvert\operatorname{Im}\lambda\rvert<\gamma_q\rho\}.$
\end{lemma}
In \cite{Biswas2019} the authors observed that the spherical Fourier transform can be extended to the whole of $L^1(X)$ by fixing a base point $x\in X$ and defining for $g\in L^1(X)$
\begin{align}\label{spherical}
\hat{g}(\lambda)=\int_X g(y)\varphi_{\lambda,x}(y)\,dy,\quad \lambda\in\R.
\end{align}

Furthermore, the authors  showed:
\begin{lemma}[\cite{Biswas2019}]\label{fourierholo}
Let $x\in X$, $1\leq p<2$, $g\in L^p(X)$ and $q>2$ with $\frac{1}{p}+\frac{1}{q}=1$ then $\hat{g}$ extends to a holomorphic function in $\lambda$ on the strip $S_q$ and is bounded on any closed sub strip $\{\lambda \in S_q\mid \lvert\operatorname{Im}\lambda\rvert\leq t\}$, for $0<t<\gamma_q\rho.$
In particular $\hat{g}$ on $\R$ satisfies the bound
$$\lVert \hat{g}\rVert_{\infty}\leq C_p\lVert g\rVert_p$$
for some constant $C_p$.
\end{lemma}

From the next proposition, it becomes clear that the name "multiplier" is justified.
 Note that in large parts this discussion carries over from the discussion of so-called Fourier multipliers on hypergroups, see for instance  \cite{fouriermulti}. 
 Furthermore a version of Proposition \ref{multiplier} for Damek-Ricci spaces was shown in \cite{Anker96} and for symmetric spaces in \cite{Stein}.
\begin{prop}[\cite{MR4443685}]\label{multiplier}
Let $1\leq q<2$ and $p>2$ such that $\frac{1}{q}+\frac{1}{p}=1$ and let $T:L^p(X)\to L^p(X)$ be an $L^p$-multiplier. Then for any radial $C_c^{\infty}$ function $\phi$ we have $T\phi\in L^q(X)$  and  there exist a holomorphic function 
$m_T$ on $S_p:=\{\lambda\in \C\mid \lvert \operatorname{Im}\lambda\rvert<(1-2/p)\rho\}$ such that for any radial $C_c^{\infty}$ function $\phi$ we have

$$\widehat{T\phi}(\lambda)=m_T(\lambda)\hat{\phi}(\lambda),\quad\lambda\in S_p.$$
\end{prop}

\begin{bem}\label{bem:multifunction}
Note that the above also holds for $T:L^q(X)\to L^q(X)$ by duality. Therefore it holds for all $L^p$-multipliers if $p\neq 2$.
\end{bem}
The next lemma is elementary but instrumental to show that $L^p$-multiplier can not be chaotic in the range $p\in[1,2)$. 		
\begin{lemma}\label{lemma:dual}
Let $\frac{1}{p}+\frac{1}{q}=1$,  $p\neq q$ and $p\neq \infty$.
For an $L^p$-multiplier  $T:L^p(X)\to L^p(X)$  with symbol $m_T(\lambda)$, its dual  $T^*:L^q(X)\to L^q(X)$ is a $L^q$ multiplier and 
has symbol $\overline{m_T(\bar{\lambda})}$.
\end{lemma}
%
\begin{proof}
Let $\psi, \phi$ be radial $C^{\infty}_c$ functions around $\sigma\in X$ and $\langle\cdot,\cdot\rangle:L^q(X)\times L^p(X)\to \C$ the canonical dual pairing.
Then by the Plancherel theorem we obtain: 
\begin{align*}
\langle T^*\psi,\phi\rangle=\langle \psi, T\phi\rangle&=\int_X  \psi(x)\cdot \overline{T\phi (x)}\,dx\\
&=\int_0^{\infty}\hat{\psi}(\lambda) \cdot\overline{\widehat{T\phi}(\lambda)}\,d\lambda \\
&=\int_0^{\infty}\overline{m_T(\lambda)}\hat{\psi}(\lambda) \cdot\overline{\hat{\phi}(\lambda)}\,d\lambda.
\end{align*}
Since $T$ is bounded, so is $\overline{m_T(\lambda)}$ which implies that $\overline{m_T(\lambda)}\hat{\psi}$ is an $L^2$-function. Now by the above $\overline{m_T(\lambda)}$ can not be orthogonal to the range of the Fourier transform. Hence there exists a radial $L^2$-function $g$  on $X$ such that 
\begin{align*}
\hat{g}(\lambda)=\overline{m_T(\lambda)}\hat{\psi}(\lambda)\quad\forall \lambda\in \R.
\end{align*}
 Therefore, 
\begin{align*}
\int_X T^*\psi(x)\cdot\overline{\phi(x)}\,dx= \int_X g(x)\cdot \overline{\phi(x)}\,dx,
\end{align*}
 hence $g=T^*\psi$ by duality. 
Now using the Fourier inversion formula we see that $T^*:L^q(X) \to L^q(X)$ is an $L^q$-multiplier. Hence, $\overline{m_T(\lambda)}$ extents to a holomorphic function on a strip around the real line as does $m_T(\lambda)$. Therefore, the holomorphic extension is given by $\overline{m_T(\bar{\lambda})}$.
\end{proof}
The next proposition shines a light on the action of $L^p$-multipliers for $p>2$ on the radial eigenfunctions $\varphi_{\lambda,\sigma}$ for $\lambda\in S_p$.
 \begin{prop}[\cite{MR4443685}]\label{prop:varphieigenfunction}
 Let $1\leq q<2$ and $p>2$ be such that $\frac{1}{p}+\frac{1}{q}=1$. Let $T:L^p(X)\to L^p(X)$ be an $L^p$-multiplier. Then for all $\lambda\in S_p$ and $x\in X$ we have
 $$T\varphi_{\lambda,x}=m_T(\lambda)\varphi_{\lambda,x}.$$

 \end{prop}

 \subsection{Chaotic Behaviour of $L^p$-multipliers for $p>2$.}
 In this section, we want to showcase the results from \cite{MR4443685} saying that the action given by the iteration of $T, $ $$T^{k}:\N\times L^p(X) \to L^p(X)$$ is chaotic in the sense of Devaney for $p>2$. 
For that purpose let us briefly recall the notion of chaos and state a lemma that is going to be of great significance in the proof of Theorem \ref{chaotic1}.
For a thorough discussion of the topic of liner chaos, we refer the reader to \cite{grosse2011linear}.
 \begin{defi}
 Let 
 $M$ be a separable Banach space over $\C$  and $A:M\to M$ be a bounded linear operator  such that $A:\N\times M\to M$ with $(n,m)\mapsto A^n(m)$ defines a dynamical system. The action of $A$ in $M$ is called chaotic if:
 \begin{enumerate}
 \item  It is topologically transitive. Namely, if for all $U,V$ non empty open subsets of $M$ there exist an $N\in\N$ such that $$A^N(U)\cap V\neq \emptyset.$$
 \item It has a  dense set of periodic points. 
 \end{enumerate}
 \end{defi}
\begin{bem}
Note that the definition of chaos can also be adapted to the action of semi-groups, and all our results hold true in this case as well by the usual adaptations from discrete to continuous dynamical systems. 
\end{bem}
 At this point let us remark that the butterfly effect, meaning sensible dependence on initial conditions, is already implied by the conditions above. See for instance Theorem 1.2.9 in \cite{grosse2011linear}. Furthermore by Birkhoff's transitivity theorem hypercyclicity, meaning the action above 
 having dense orbits is equivalent to topological transitivity. 
 The following theorem by Godfey and Shapiro 

 gives a useful characterisation of the topologically transitive action of operators on Banach spaces. 
 \begin{satz}[\cite{grosse2011linear} Theorem 3.1 ]\label{prop:eigenfunctions}

  \label{Godfrey}
 Let $M$ be a separable Banach space over $\C$  and let $A:M\to M$ be a continuous linear 
 operator. Suppose the subsets $M^+,M^-$ defined by 
 \begin{align*}
 M^+&=\operatorname{Span}\{v\in M\mid Av=\lambda v,\lambda \in  \C,\lvert \lambda\rvert>1\}\\
M^-&=\operatorname{Span}\{v\in M\mid Av=\lambda v,\lambda \in  \C,\lvert \lambda\rvert<1\}
\end{align*}
are dense in M. Then the action of $T$ on M, given by iteration,  is mixing i.e.  there exists an $N\in\N$ such that 
\begin{align*}A^n(U)\cap V\neq \emptyset \quad \forall n\geq N, 
\end{align*}
and especially the action is topologically transitive.
\end{satz}
\begin{bem}
Note at this point that for the notion of chaos, the dimension of the underlying space has to be infinite since there are no  topologically transitive 
linear operators on finite dimensional vector spaces \cite[Theorem 2.5.8]{grosse2011linear}.
\end{bem}
\begin{lemma}[\cite{MR4443685}]\label{lemma:dens}
Let $2<p<\infty ,\sigma\in X$ and $K\subset S_p$ such that $K$ has a limit point in $S_p$, then 
\begin{align*}
V:=\operatorname{Span}\{\tau_x\varphi_{\lambda,\sigma}\mid x\in X,\lambda\in K\}
\end{align*}
is dense in $L^p(X)$.
\end{lemma}

\begin{lemma}[\cite{MR4443685}]\label{lemma:nonconstant}
Let $2<p<\infty$ and $T$  an $L^p$-multiplier that is not a multiple of the identity, then the symbol $m_T$ is not constant. 
\end{lemma}
\begin{proof}
Suppose we have an $L^p$-multiplier $T$ such that $m_T=C$ for some constant $C\in \C$. Choose $K\subset S_p$ and define $V$ as in Lemma \ref{lemma:dens}.
 Then by Proposition \ref{prop:varphieigenfunction}
 \begin{align}
 T\varphi_{\lambda,x}= C\varphi_{\lambda,x}
 \end{align}
 for all $\lambda\in S_p$ and $x\in X$. Hence, $T=C\operatorname{id}$ on $V$. Then by the density of $V$ in $L^p(X)$, we also have $T=C\operatorname{id}$ on the whole of $L^p(X)$. Hence for every $L^p$-multiplier that is not a multiple of the identity $m_T$ is not constant. 
\end{proof}

\begin{satz}[\cite{MR4443685}]\label{chaotic1}
Let $X$ be a simply connected harmonic manifold of rank one with mean curvature of the horopheres $2\rho$. Let $2<p<\infty$ and let $T:L^p(X)\to L^p(X)$ be a $L^p$-multiplier with symbol $m_T$ such that $T$ is not a scalar multiple of the identity. Where the symbol is defined by
$$
\widetilde{Tf}^{\sigma}(\lambda,\xi)=m_T(\lambda)\widetilde{f}^{\sigma}(\lambda,\xi)\quad\
$$
for all smooth radial functions $f$ with compact support. 
Then for all $\lambda\in S_p:=\{\lambda\in \C\mid \lvert \operatorname{Im}\lambda\rvert<(1-2/p)\rho\}$ for which $m_T(\lambda)\neq 0$ and for any $\nu\in \C$ for which $\lvert \nu\rvert =\lvert m_T(\lambda)\rvert$ the dynamics of the operator $\frac{1}{\nu}T$ on $L^p(X)$ are chaotic in the sense of Devaney. 
\end{satz}

\subsection{Non Chaotic Behaviour of $L^p$-multipliers for $p\leq2$}

After showcasing the arguments for the fact that suitably scaled $L^p$-multipliers exhibit a chaotic behaviour for $p>2$ we now show that this bound is sharp under the assumption that the bounded linear operator $T:L^p(X)\to L^p(X)$ acts on the $C^{\infty}_c$ functions by multiplication with a holomorphic, for $p<2$, or continuous, for $p=2$ function $m_{T}(\lambda)$ under Fourier transformation.
For this, we first need to discuss the domain of holomorphy of the Fourier transform for $L^p$-function with $1\leq p<2$ which is the main argument for the proof of Theorem \ref{thm:L^ppleq2} in this range. The case $p=2$ will be dealt with separately using the Plancherel theorem for the Fourier transform. 
\begin{bem}
In this section, it will be more convenient to describe the topological transitivity of a linear operator $A$  by the equivalent property of $A$ having a dense orbit. Such an operator is then called hypercyclic and an element with a dens orbit is called hypercyclic vector. The equivalence of transitivity and hypercyclicity is proven in Birkhoff's transitivity theorem see  \cite[Theorem 2.19]{grosse2011linear}.
\end{bem}

\begin{lemma}\label{lemma:holo2}
Let $\sigma\in X$ and for $1\leq p<2$ let $f\in L^p(X)$ and $\frac{1}{p}+\frac{1}{q}=1$. Then
\begin{align*}
\widetilde{f}^{\sigma}(\lambda,\xi)=\int_X f(x)e^{(-i\lambda-\rho)B_{\xi,\sigma}(x) }\,dx
\end{align*}
exists for all $\lambda \in S_q$ and almost every $\xi\in\partial X$ and is holomorphic on $S_q$. Where the set of full measure in $\partial X$ on which $\widetilde{f}^{\sigma}$ is defined only depends on $f$. 
\end{lemma} 
\begin{proof}
Fix $\sigma\in X$ then we have by the definition of the Fourier transform: 
\begin{align}\label{eq:holo}
\int_{\partial X} \lvert \widetilde{f}^{\sigma}(\lambda,\xi)\rvert \,d\mu_{\sigma}(\xi) \leq \int_X \lvert f(x) \rvert \int_{\partial X} e^{(-i\lambda-\rho)B_{\xi,\sigma}(x)}\,d\mu_{\sigma}(\xi) \,dx.
\end{align}
Additionally we have by the representation of the eigenfunctions: 
\begin{align*}
\int_{\partial X} e^{(-i\lambda-\rho)B_{\xi,\sigma}(x)}\,d\mu_{\sigma}(\xi)=\varphi_{-\lambda,\sigma}=\varphi_{\lambda,\sigma}.
\end{align*}
Now  by Lemma \ref{lemma2.4.9} we have for $\lambda\in S_q$ 
and  $t=\operatorname{Im}\lambda$, 
\begin{align*}
\lVert \varphi_{\lambda,x}\rVert_q\leq \lVert \varphi_{it,x}\rVert_q<\infty.
\end{align*}
Therefore we conclude, with the H\"older inequality, that (\ref{eq:holo}) is bounded by 
\begin{align*}
\lVert f\rVert_p\cdot\lVert \varphi_{\lambda,\sigma}\rVert_q<\infty,
\end{align*}
which yields the first part of the assertion. For the second part observe that $\varphi_{\lambda,\sigma}$ depends holomorphically on $\lambda$ for $\lambda\in S_q$ as does $e^{(-i\lambda-\rho)B_{\xi,\sigma}(x)}$. Now since the integral above is bounded for almost every $\xi\in\partial X$ and every $\lambda\in S_{q}$ and the integrant is holomorphic it follows that the integral of $ \widetilde{f}^{\sigma}$ over every piecewise $C^{1}$ curve in $S_{q}$ vanishes, hence  $ \widetilde{f}^{\sigma}$  is holomorphic by Morera's theorem.
\end{proof} 
\begin{lemma}\label{lemma:holo}
Let $\sigma\in X$ be fixed. 
For $f\in L^p(X)$ with $1\leq p< 2$ and $\lambda\in(0,\infty)$ we have:
\begin{align*}
\lVert \widetilde{f}^{\sigma}(\lambda,\cdot)\rVert_{L^2(\partial X,d\mu_{\sigma})}\leq C\lVert f\rVert_p
\end{align*}
for some constant $C>0$ not depending on $\lambda$ or $f$. 
\end{lemma}
\begin{proof}
Fix $\sigma\in X$. Then
\begin{align}
\nonumber\int_{\partial X} \lvert \widetilde{f}^{\sigma}(\lambda,\xi)\rvert^2 &\,d\mu_{\sigma}(\xi)=\int_{\partial X}\overline{ \widetilde{f}^{\sigma}(\lambda,\xi)}  \widetilde{f}^{\sigma}(\lambda,\xi)\,d\mu_{\sigma}(\xi) \\\nonumber
&=\int_{\partial X} \int_X \overline{ f(x)e^{(-i\lambda-\rho)B_{\xi,\sigma}(x) }}\,dx\widetilde{f}^{\sigma}(\lambda,\xi)\,d\mu_{\sigma}(\xi)\\
&=\int_X \overline{f(x)}\int_{\partial X}\overline{e^{(-i\lambda-\rho)B_{\xi,\sigma}(x) }} \widetilde{f}^{\sigma}(\lambda,\xi)\,d\mu_{\sigma}(\xi)\,dx.\label{eq:holo2}
\end{align}
Let $q$ be the H\"older conjugate to $p$ and consider 
\begin{align*}
f*\varphi_{\lambda,\sigma}(x)&=\int_X f(y) \varphi_{\lambda,x}(y)\,dy\\
&=\int_X f(y)\int_{\partial X} e^{(-i\lambda-\rho)B_{\xi,x}(y) }\,d\mu_{x}(\xi)\,dy\\
&=\int_X \int_{\partial X} f(y)e^{(-i\lambda-\rho)B_{\xi,x}(y) }\,d\mu_{x}(\xi)\,dy\\
&=\int_{\partial X} \widetilde{f}^{x}(\lambda,\xi)\,d\mu_x(\xi)\\
&=\int_{\partial X} \widetilde{f}^{\sigma}(\lambda,\xi) e^{(i\lambda-\rho)B_{\xi,\sigma}(x)}\,d\mu_{\sigma}(\xi).
\end{align*} 
Since $\lambda\in\R$, we 
get for  equation (\ref{eq:holo2})
\begin{align*}
&\int_X \overline{f(x)}\int_{\partial X}\overline{e^{(-i\lambda-\rho)B_{\xi,\sigma}(x) }} \widetilde{f}^{\sigma}(\lambda,\xi)\,d\mu_{\sigma}(\xi)\,dx\\
&= \int_X \overline{f(x)}f*\varphi_{\lambda,\sigma}(x)\,dx\\
&\leq \lVert f\rVert_p\cdot\lVert f*\varphi_{\lambda,\sigma}\rVert_q.
\end{align*}
 The inequality is due to the H\"older inequality. Since by Lemma \ref{lemma2.4.9} $\varphi_{\lambda,\sigma}\in L^{2+\epsilon}(X)$ and the Young inequality (Lemma  \ref{young}) for every $\epsilon>0$, the convolution with $ \varphi_{\lambda,\sigma}$ defines a bounded linear operator from $L^p(X)$ to $L^q(X)$ with operator norm bounded by $\lVert \varphi_{0,x}\rVert_q$, we conclude 
\begin{align*}
\int_{\partial X} \lvert \widetilde{f}^{\sigma}(\lambda,\xi)\rvert^2 \,d\mu_{\sigma}(\xi)\leq C\lVert f\rVert^2_p
\end{align*}
for some constant $C>0$. 
\end{proof}
From now on we assume that for a bounded linear operator $T:L^p(X)\to L^p(X)$  we have:
\begin{align}\label{con:T=m}
\widetilde{Tf}^{\sigma}(\lambda,\xi)=m_T(\lambda)\widetilde{f}^{\sigma}(\lambda,\xi)\quad\forall f\in C_c^{\infty}(X),
\end{align} 
where, for $1\leq p <2$ and $\frac{1}{p}+\frac{1}{q}=1$, $m_T$ is a holomorphic function on the strip $S_q$ and continuous on $\R$ for $p=2$.
Note that this includes all $L^p$-multipliers arising from the convolution with a radial function. 
\begin{lemma}\label{lemma:ae. multi}
For $1\leq p<q<2$ let $T$ be an $L^p$-multiplier satisfying condition (\ref{con:T=m}) above with symbol $m_T$ and $f\in L^p(X)$. Then there exists a subset $B\subset \partial X$ of full measure such that for every $\xi\in B$ and $\lambda\in S_q:=\{\lambda\in \C\mid \lvert \operatorname{Im}\lambda\rvert<(\frac{2}{q}-1)\}$ we have 
\begin{align*}
\widetilde{Tf}^{\sigma}(\lambda,\xi)=m_T(\lambda)\widetilde{f}^{\sigma}(\lambda,\xi).
\end{align*} 
\end{lemma} 
\begin{proof}
Fix $\sigma\in X$.
Since $C^{\infty}_c(X)$ is dens in $L^p(X)$, there is a sequence $f_n\in C^{\infty}_c(X)$ converging to $f$ in $L^p(X)$. Hence, by Lemma  \ref{lemma:holo} we find a sub-sequence again denoted by $f_n$ such that $\widetilde{f_n}^{\sigma}(\lambda,\xi)\to \widetilde{f}^{\sigma}(\lambda,\xi)$ for every $\lambda\in (0,\infty)$ and almost every $\xi\in \partial X$. Since $T$ is bounded we also have $Tf_n\to Tf$ in $L^p(X)$. After passing to a sub-sequence if necessary  we find that  $\widetilde{Tf_n}^{\sigma}(\lambda,\xi)\to \widetilde{Tf}^{\sigma}(\lambda,\xi)$ for a fixed  $\lambda\in(0,\infty)$ and almost every $\xi\in\partial X$. By the assumption that $T$ satisfies (\ref{con:T=m})   we have for $\lambda$ and $\xi$ as above:
\begin{align*}
\widetilde{Tf_n}^{\sigma}(\lambda,\xi)=m_T(\lambda)\widetilde{f_n}^{\sigma}(\lambda,\xi).
\end{align*}
Hence, for those $\lambda$ and $\xi$ 
\begin{align*}
\widetilde{Tf_n}^{\sigma}(\lambda,\xi)\to m_T(\lambda)\widetilde{f}^\sigma(\lambda,\xi),\text{ for } n\to \infty.
\end{align*}

 Now observe that by Lemma \ref{lemma:holo2} there is a set of full measure $B^{\prime}\subset \partial X$ such that for every fixed $\xi\in B^{\prime}$ the maps $\lambda\mapsto \widetilde{f}(\lambda,\xi)$ and $\lambda\mapsto \widetilde{Tf}(\lambda,\xi)$ are holomorphic on $S_q\subset S_p$. Therefore, the equality 
 \begin{align*}
 m_T(\lambda)\widetilde{f}^\sigma(\lambda,\xi)=\widetilde{Tf}^\sigma(\lambda,\xi)
 \end{align*}
  extends to the whole of $S_q$. 
\end{proof}
The following general observation on hypercyclic operators is going to be instrumental in the proof of  Theorem \ref{thm:L^ppleq2}. For a proof see \cite[Proposition 5.1 and Lemma 2.53]{grosse2011linear}.
\begin{lemma}[\cite{grosse2011linear}]\label{lemma:chos2}
Let $A:M\to M$ be a hypercyclic operator on the Banach space $M$ and $A^*:M^*\to M^*$ its dual. Then:
\begin{enumerate}
\item for every non zero $m\in M^*$ the orbit $\{(A^*)^n m\mid n\geq 0\}$ is unbounded,
\item the point spectrum of $A^*$ is empty.
\end{enumerate}
\end{lemma}

\begin{satz}\label{thm:L^ppleq2}
Let $X$ be a non-compact simply connected harmonic manifold of rank one and $\sigma\in X$. 
Let $1\leq p\leq2$ and $T:L^p(X)\to L^p(X)$ be $L^p$-multiplier such that for all $\lambda\in \R$ and all $\xi\in\partial X$ we have:
\begin{align*}
\widetilde{Tf}^{\sigma}(\lambda,\xi)=m_T(\lambda)\widetilde{f}^{\sigma}(\lambda,\xi)\quad\forall f\in C_c^{\infty}(X),
\end{align*}
where for $1\leq p <2$ and $\frac{1}{p}+\frac{1}{q}=1$, $m_T$ is a holomorphic function on the strip $S_q$ and continuous on $\R$ for $p=2$. Then $T$ is neither hypercyclic nor has periodic points, hence is not chaotic. 

\end{satz}
We will split the proof of Theorem \ref{thm:L^ppleq2} into two propositions, the first one dealing with the case $p<2$ and the second one with the case $p=2$. 
\begin{prop}\label{prop:nochoasp<2}
Fix $1\leq p<2$ and $q$ its H\"older conjugate and let $T:L^p(X)\to L^p(X)$ be a nontrivial $L^p$-multiplier satisfying (\ref{con:T=m}). Then $T$ is neither hypercyclic nor has periodic points, hence is not chaotic. 
\end{prop}
\begin{proof}
Fix $\sigma\in X$. Recall that by Proposition \ref{prop:varphieigenfunction} every $\varphi_{\lambda,\sigma}$ with $\lambda\in S_p^{\circ}$ is an eigenfunction of $T^*$. Hence, by Lemma \ref{lemma:chos2} $T$ can not be hypercyclic. Now suppose that $T$ has symbol $m_T(\lambda)$ and fix $p^{\prime}\in(p,2)$. If there is a non zero function $g\in L^p(X)$ with $T^ng=g$ for some $n\in\N\setminus\{0\}$ then by Lemma \ref{lemma:ae. multi}
\begin{align*}
(m_T(\lambda)^n-1)\widetilde{g}^{\sigma}(\lambda,\xi)=0
\end{align*}
for $\lambda\in S_{p^{\prime}}$ and almost every $\xi\in \partial X$. 
Since  by Lemma \ref{lemma:holo2} for almost every $\xi\in\partial X$, $\widetilde{g}(\lambda,\xi)$ is holomorphic on $S_{p^{\prime}}$,  it can only vanish on a set of measure zero in $\C$. Therefore $m_T(\lambda)^n=1$ on $S_{p^{\prime}}$ that is $\lvert m_T(\lambda)\rvert=1$. But since $m_T$ is holomorphic, this contradicts the open mapping theorem.
\end{proof}
\begin{prop}\label{prop:nochoasp=2}
Let $T:L^2(X)\to L^2(X)$ be a nontrivial $L^2(X)$ multiplier fulfilling the conditions of Theorem \ref{thm:L^ppleq2}. Then $T$ is not hypercylic and hence not chaotic.
\end{prop}
\begin{proof}
Fix $\sigma\in X$. Let $m_T$ 
 be the symbol of $T$. Then for $f\in L^2(X)$ 
\begin{align*}
\widetilde{Tf}^{\sigma}(\lambda,\xi)=m_T(\lambda) \widetilde{f}^{\sigma}(\lambda,\xi)
\end{align*}
 by density of $C^{\infty}_c(X)$ in $L^2(X)$. For the essential supremum of $m_T$ we have:
 \begin{align*}
 \lVert m_T\rVert_{\infty}
 =  \lVert T\rVert_{L^2}<\infty,
 \end{align*}
 where $\lVert \cdot\rVert_{L^2}$ denotes the operator norm.
Assume $T$ to be hypercyclic and $\phi$ a hypercyclic vector. Then there is a sequence $n_k$ of natural numbers such that $T^{n_k}\phi\to 2\phi$ in $L^2(X)$. 
This implies $\lVert T^{n_k}\phi\rVert_2\to 2\lVert \phi\rVert_2$. 
Hence
\begin{align*}
\lim_{n_k\to \infty}&\int_{0}^{\infty}\int_{\partial X} \lvert m_T (\lambda)\rvert^{2n_k} \lvert\widetilde{\phi}^{\sigma}(\lambda,\xi) \rvert^2\lvert \mathbf{c}(\lambda)\rvert^{-2}\,d\mu_{\sigma}(\xi)\,d\lambda\\&=4\int_{0}^{\infty}\int_{\partial X} \lvert\widetilde{\phi}^{\sigma}(\lambda,\xi) \rvert^2\lvert \mathbf{c}(\lambda)\rvert^{-2}\,d\mu_{\sigma}(\xi)\,d\lambda.
\end{align*}
By separating the integral on the left we obtain:
\begin{align*}
&\lim_{n_k\to \infty}\int_{\{\lambda\in (0,\infty)\mid m_T(\lambda)<1\}\times\partial X}
 \lvert m_T (\lambda)\rvert^{2n_k} \lvert\widetilde{\phi}^{\sigma}(\lambda,\xi) \rvert^2\lvert \mathbf{c}(\lambda)\rvert^{-2}\,d\mu_{\sigma}(\xi)\,d\lambda\\
 &+\lim_{n_k\to \infty}\int_{\{\lambda\in (0,\infty)\mid m_T(\lambda)=1\}\times\partial X}
 \lvert m_T (\lambda)\rvert^{2n_k} \lvert\widetilde{\phi}^{\sigma}(\lambda,\xi) \rvert^2\lvert \mathbf{c}(\lambda)\rvert^{-2}\,d\mu_{\sigma}(\xi)\,d\lambda\\
 &+\lim_{n_k\to \infty}\int_{\{\lambda\in (0,\infty)\mid m_T(\lambda)>1\}\times\partial X}
 \lvert m_T (\lambda)\rvert^{2n_k} \lvert\widetilde{\phi}^{\sigma}(\lambda,\xi) \rvert^2\lvert \mathbf{c}(\lambda)\rvert^{-2}\,d\mu_{\sigma}(\xi)\,d\lambda\\
&\quad=4\int_{0}^{\infty}\int_{\partial X} \lvert\widetilde{\phi}^{\sigma}(\lambda,\xi) \rvert^2\lvert \mathbf{c}(\lambda)\rvert^{-2}\,d\mu_{\sigma}(\xi)\,d\lambda.
\end{align*}
 By dominant convergence we get
 
\begin{align*}
\lim_{n_k\to \infty}&\int_{\{\lambda\in (0,\infty)\mid m_T(\lambda)<1\}\times\partial X}
 \lvert m_T (\lambda)\rvert^{2n_k} \lvert\widetilde{\phi}^{\sigma}(\lambda,\xi) \rvert^2\lvert \mathbf{c}(\lambda)\rvert^{-2}\,d\mu_{\sigma}(\xi)\,d\lambda=0.
 \end{align*}
 This implies 
 \begin{align}
 \lim_{n_k\to \infty}&\int_{\{\lambda\in (0,\infty)\mid m_T(\lambda)>1\}\times\partial X}
 \lvert m_T (\lambda)\rvert^{2n_k} \lvert\widetilde{\phi}^{\sigma}(\lambda,\xi) \rvert^2\lvert \mathbf{c}(\lambda)\rvert^{-2}\,d\mu_{\sigma}(\xi)\,d\lambda\label{eq:L^2choas1}\\
 &\leq 4\int_{0}^{\infty}\int_{\partial X} \lvert\widetilde{\phi}^{\sigma}(\lambda,\xi) \rvert^2\lvert \mathbf{c}(\lambda)\rvert^{-2}\,d\mu_{\sigma}(\xi)\,d\lambda\label{eq:L^2choas2}.
 \end{align}
 But by the monotonous convergence theorem, (\ref{eq:L^2choas1}) tends to infinity while (\ref{eq:L^2choas2})  is finite. Hence, either $\{\lambda\in (0,\infty)\mid m_T(\lambda)>1\}$ is a zero set or $\widetilde{\phi}^{\sigma}$ vanishes on it.

 The first means that $T$ is a contraction or the identity and the second implies by the Plancherel theorem that $\lVert T\phi\rVert_2 \leq \lVert \phi\rVert_2$ with both options contradicting the assumption. 
 
\end{proof}
\begin{prop}\label{prop:p=2noperiodic}
Let $T:L^2(X)\to L^2(X)$ be a nontrivial $L^2(X)$ multiplier fulfilling the conditions of Theorem \ref{thm:L^ppleq2}, such that $\lvert m_T(\lambda)\rvert$ is not constant, then $T$ has no periodic points. 
\end{prop}
\begin{proof}
Assume that $T$ has a non trivial periodic point $g\in L^2(X)$ then there is an $n\in \N$ such that 
\begin{align*}
(m_T(\lambda)^n-1)\widetilde{g}^{\sigma}(\lambda,\xi)=0
\end{align*}
almost everywhere. 
But then either  $\lvert m_T(\lambda)\rvert=1$ for every $\lambda\in \R$ or $g=0$ a.e by the inverse Fourier transform. Both contradict the assumptions. 
\end{proof}

This completes the proof of  Theorem \ref{thm:L^ppleq2}.

\section{Examples/Discussion}
\subsection{Dynamics of The Heat Semi-group}\label{sec:chosheat}

 The Laplacian on the space $L^2(X)$ is essentially self-adjoint and negative, hence it generates a semi-group by the spectral theorem for unbounded self-adjoint operators. Furthermore $e^{t\Delta}$ is positive and leaves $L^1(X)\cap L^{\infty}(X)\subset L^2(X)$ invariant and therefore can be extended to a positive continuous semi-group on $L^p(X)$ for $p\in [1,\infty]$ which is strongly continuous for $p\in [1,\infty)$. With convolution kernel $h_t$ called the heat kernel. This is a function on $\R^+\times X\times X$ and $h_t(x,\cdot)\in L^p(X)$ for $p\in[1,\infty]$. See \cite{davies1990heat} for details. Szabo \cite{szabo1990} showed that on a complete simply connected manifold $h_t$ is radial if and only if the manifold is harmonic. Since we again assume that $(X,g)$ is a non-compact simply connected harmonic manifold of rank one, in abuse of notation we will write $h_t(x,y)=h_t(r)$ where $r=d(x,y)$. 
 \begin{satz}\label{hyperkernel}
Let $X$ be a simply connected non-compact harmonic manifold of rank one with mean curvature of the horopheres $2\rho$. Let $\sigma\in X$. The Fourier transform of the heat kernel $h_t(\sigma,\cdot)$ around $\sigma$ is given by 
\begin{align*}
\hat{h}_t^{\sigma}(\lambda)=e^{-t(\lambda^2+\rho^2)}\quad\forall \lambda\in \C
\end{align*}
therefore for all $p\geq 1$ and $t>0$ the heat semi-group $e^{t\Delta}:L^p(X)\to L^p(X)$ defines an $L^p$-multiplier with symbol $e^{-t(\lambda^2+\rho^2)}$.
Furthermore, the heat kernel has the integral representation:
\begin{align*}
h_t(x,y)=C_0 \int_{0}^{\infty}e^{-t(\lambda^2 +\rho^2)}\varphi_{\lambda,x}(y)\vert \mathbf{c}(\lambda)\vert^{-2}\,d\lambda, \forall x,y\in X
\end{align*}
 \end{satz}
 \begin{proof}
 Let $x,y,\sigma\in X$. Consider a radial eigenfunction of the Laplacian  $\varphi_{\lambda,\sigma}$ with $\Delta \varphi_{\lambda,\sigma}=-(\lambda^2+\rho^2)\varphi_{\lambda,\sigma}$
 then 
 \begin{align*}
 \frac{\partial}{\partial t} e^{-t(\lambda^2+\rho^2)}\varphi_{\lambda,\sigma}=\Delta \varphi_{\lambda,\sigma}
 \end{align*}
 hence $u(t,x):=e^{-(\lambda^2+\rho^2)}\varphi_{\lambda,\sigma}(x)$ is a solution of the heat equation with $u(0,x)=\varphi_{\lambda,\sigma}(x)$.
 Therefore we have, since $h_t$ is the fundamental solution of the heat equation:
 \begin{align*}
 \hat{h}_t^{\sigma}(\lambda)&=\int_X \varphi_{\lambda,\sigma}(x)h_t(\sigma,x)\,dx\\
 &=u(t,\sigma)\\
 &=e^{-(\lambda^2+\rho^2)}\varphi_{\lambda,\sigma}(\sigma)\\
 &=e^{-(\lambda^2+\rho^2)}
 \end{align*}
 Showing the first two assertions. 
The second part follows now by applying the inverse Fourier transform to the above.
 \end{proof}
Note that Theorem \ref{hyperkernel} can also be deduced from \cite{PS15}.   

\begin{satz}[\cite{MR4443685}]\label{thm:heatchaos}
Let $X$ be a simply connected non-compact harmonic manifold of rank one with mean curvature of the horopheres $2\rho$ and let $2<p<\infty$ and $1<q<2$  be such that $\frac{1}{p}+\frac{1}{q}=1$. Then there exists a constant $c_p:=\frac{4\rho^2}{pq}$ such that the action of the shifted heat semi-group $(e^{ct}e^{t\Delta})_{t>0}$ on $L^p(X)$ is chaotic in the sense of Devaney for all $c\in \C$ with $\operatorname{Re}c>c_p$. In fact, for any $t_0>0$ the operator $(e^{ct_0}e^{t_0\Delta})$ on $L^p(X)$ is chaotic, in the sense that the induced dynamical system is chaotic, for all $c\in \C$ with $\operatorname{Re} c>c_p$.
\end{satz}

Next, we are going to see why the conditions $p>2$ in Theorem \ref{thm:heatchaos} is sharp. 

\begin{satz}\label{thm:nochoas2}
Let $X$ be a simply connected non-compact harmonic manifold of rank one, $1\leq p\leq 2$. Then for every $c\in \C$,
$$e^{ct_0}e^{t_0\Delta}:L^p(X)\to L^p(X)$$ is not chaotic. 
In particular, for $1\leq p <2$ it  is not hypercyclic and does not have dens periodic points. Further   $T$ can not have periodic points in $L^2(X)$. 
\end{satz}\label{thm:c_p}
\begin{proof}
By Theorem \ref{hyperkernel} we have that for every $c\in \C$ $e^{ct_0}e^{t_0\Delta }$ satisfied condition (\ref{con:T=m}) and since the heat kernel is a radial function we have that $e^{ct_0}e^{t_0\Delta }$  is a $L^p$-multiplier. Therefore it satisfies the conditions of Theorem \ref{thm:L^ppleq2} for $1\leq p< 2$. This yields the claim in this range. 
And since $\lvert e^{ct_0} e^{-(\lambda^2+\rho^2)t_0}\rvert$ is not constant in $\lambda$,  $e^{ct_0}e^{t_0\Delta}$ also satisfies the conditions of Theorem \ref{thm:L^ppleq2} for $p=2$ and  by \ref{prop:p=2noperiodic} has no periodic points in this case.
\end{proof}

We now want to show that $c_p$ in Theorem \ref{thm:heatchaos} is optimal.
\begin{lemma}\label{lemma:spectralbound}
Let $p\neq 2$, fix $t>0$ and define $T:=e^{t\Delta}:L^p(X)\to L^p(X)$. Then the operator norm $\lVert T\rVert_{L^p}$ is bounded by $\hat{h}_t(-i\gamma_p\rho)$.
\end{lemma}
\begin{proof}
Assume $p>2$ since then $p<2$ follows from duality. By Lemma \ref{lemma:dens} the set $V:=\operatorname{Span}\{\tau_x\varphi_{\lambda,\sigma}\mid x\in X,\lambda\in S_p\}$ is dense in $L^p(X)$ and by \cite{10.1215/S0012-7094-89-05836-5} the spectrum of the Laplacian on $L^p(X)$ is contained in $S_p$.
Furthermore by Proposition \ref{prop:varphieigenfunction} and Theorem \ref{hyperkernel} we have that  $T\varphi_{\lambda,\sigma}=e^{-t(\lambda^2+\rho^2)} \varphi_{\lambda,\sigma}$.

Therefore we can conclude:
\begin{align*}
\lVert T \rVert_{L^p}&\leq\sup_{\lambda\in S_p} e^{-t(\lambda^2+\rho^2)}\\
&=\sup_{\lambda\in S_p} \hat{h}_t(\lambda)\\
&\leq \hat{h}_t(-i\gamma_p\rho)\\
&=e^{-c_pt}.
\end{align*}
\end{proof}
\begin{satz}
Let $1\leq p\leq \infty$ and $c\in\C$ with $\operatorname{Re}(c)\leq c_p$ then $e^{ct}e^{t\Delta}:L^p(X)\to L^p(X)$ can not be chaotic. In particular:
\begin{enumerate}
\item $e^{ct}e^{t\Delta}$ is not hypercylic if $\operatorname{Re}(c) \leq c_p$.
\item $e^{ct}e^{t\Delta}$ has no periodic points if $\operatorname{Re}(c)<c_p$. 
\end{enumerate}
\end{satz}
\begin{proof}
Assume $\operatorname{Re}(c) \leq c_p$.
We may assume $1\leq p<\infty$ since $h_t$ is a continuous function and therefore the convolution with any multiple of $e^{ct}\cdot h_t$ can not define a hypercyclic operator on $L^{\infty}(X)$. 
Now we have: $\hat{h}_t(-i\gamma_p\rho)=e^{-c_p t}$. Hence by the Lemma \ref{lemma:spectralbound} we have:
\begin{align}
\lVert e^{ct}e^{t\Delta}\rVert_{L^p\to L^p}\leq e^{\operatorname{Re} (c)t}e^{-c_p t} \leq 1.
\end{align}
Hence all orbits are bounded and therefore $e^{ct}e^{t\Delta}$ can not be hypercyclic. 
Now suppose $\operatorname{Re}(c)<c_p$ and $1\leq p\leq \infty$, then we have:
\begin{align}
\lVert e^{ct}e^{t\Delta}\rVert_{L^p\to L^p}\leq e^{\operatorname{Re} (c)t}e^{-c_p t} < 1.
\end{align}
Therefore $e^{ct}e^{t\Delta}$ is a contraction and can not have periodic points. 
\end{proof}
To conclude this section we remark that there are similar results for symmetric spaces and harmonic $NA$ groups \cite{ji_weber_2010} \cite{Sarkar_2013}.
\subsection{The Resolvent}\label{sec:resolvent}
 In this section, we present an example of an $L^p$-multiplier that does not need scaling to be chaotic in the form of the resolvent near the spectrum. For this purpose, we first need to generalise some results from \cite{resolvent} to harmonic manifolds of rank one. 
Let $(X,g)$ be a simply connected harmonic manifold of rank one. $\Delta:L^p(X)\to L^p(X)$ the $p$-Laplacian. For $z\in\C$ the resolvent is given by $R:(z)=(\Delta-z\operatorname{id})^{-1}$.
\begin{satz}\label{thm:reschaos}
Let $p>2$ then there is a $z\in \C$ such that the action of $R(z):L^p(X)\to\L^p(X)$ is chaotic. 
\end{satz}
\begin{lemma}\label{lemma:heatintegral}
We have $R(z)=-\int_0^{\infty} e^{-zt}e^{t\Delta}$ if the integral on the right-hand side converges as a formal integral. 
\end{lemma}

\begin{proof}
Let $f\in L^p(X)$ assume the integral converges formally then $g:=\int_0^{\infty} e^{-zs}(e^{s\Delta }f)\,ds$ converges point-wise. 
Now we have: 
\begin{align*}
\frac{1}{h}(e^{h\Delta} -\operatorname{id})g&=\frac{1}{h}(e^{h\Delta}-\operatorname{id})\int_0^{\infty} e^{-zs}(e^{s\Delta}f)\,ds\\
&=\frac{1}{h}\int_0^{\infty}e^{-zs}e^{(s+h)\Delta}f-e^{-zs}e^{s\Delta}f\,ds\\
&=\frac{e^{hs}}{h}\int_h^{\infty}e^{-zs}e^{s\Delta}f\,ds-\frac{1}{h}\int_0^{\infty} e^{-zs}e^{s\Delta}\,ds\\
&=\frac{e^{zh}-1}{h}\int_h^{\infty}e^{-zs}e^{s\Delta}f\,ds-\frac{1}{h}\int_0^{h}e^{-zs}e^{s\Delta}f\,ds
\end{align*}
Taking the limit $h\to 0$ implies:
\begin{align*}
\Delta g&=\lim_{h\to 0} \frac{1}{h}(e^{h\Delta}-\operatorname{id})g\\
&=z\int_0^{\infty}e^{-zs}e^{s\Delta}f\,ds-f
\end{align*}
Hence we have $\Delta g=zg+f$. This yields
\begin{align*}
f&=(\Delta g -z\operatorname{id}g)\\
&=(\Delta-z\operatorname{id})\Bigl(-\int_0^{\infty} e^{-zs}(e^{s\Delta }f)\,ds\Bigr).
\end{align*}
\end{proof}

\begin{lemma}
Let $p>2$ and $\frac{1}{p}+\frac{1}{q}=1$ then for $z\in \C$ with $\operatorname{Re} z >-c_p$ the resolvent $R(z):L^p(X)\to L^p(X)$ is bounded and by duality also $R(z):L^q(X)\to L^q(X)$.
\end{lemma}

\begin{proof}
Assume $\operatorname{Re} z >-c_p$. 
Let $f,g\in C^{\infty}_c$ and $\langle\cdot,\cdot\rangle:L^p(X)\times L^q(X)\to\C$ be the canonical dual pairing then by H\"older's inequality and Lemma \ref{lemma:spectralbound}:
\begin{align*}
\lvert \langle R(z)f,g\rangle&=\Bigl \lvert \int_X \int_0^{\infty} e^{-zt}\cdot (e^{t\Delta}f)(x)g(x)\,dt\,dx\Bigr \rvert\\
&\leq \int_0^{\infty}e^{\operatorname{Re}zt}e^{-c_pt}\lVert f\rVert_p\cdot \lVert g\rVert_q\,dt\\
&<\infty.
\end{align*}
Therefore $\lVert R(z)\rVert_{L^p\to L^p}<\infty$. 
\end{proof}

\begin{lemma}
Let $p\neq 2,a\in \R$ and $R(z):L^p(X)\to L^p(X)$ is bounded then $a>-c_p$. 
\end{lemma}

\begin{proof}
Assume $p>2$ and $\frac{1}{p}+\frac{1}{q}=1$ and $a\in \R$ then by Lemma \ref{lemma2.4.9} we have for $g\in L^q(X)$
\begin{align*}
\infty>\langle g,R(a)\varphi_{\lambda,x}\rangle&=\int_X g(y)\int_0^{\infty} e^{-at}(e^{t\Delta}\varphi_{\lambda,x})(y)\,dt \,dy\\
&=\int_X g(y)\varphi_{\lambda,x}(y)\,dy\cdot \int_0^{\infty} e^{-at}e^{-t(\lambda^2+\rho^2)}\,dt.
\end{align*}
Now by Lemma \ref{lemma:spectralbound} we have that $a\geq -c_p$ and since the resolvent set is open the inequality is strict. 
\end{proof}

\begin{folg}\label{folg;resolv1}
Let $a\in \R$, $p>2$ and $\frac{1}{p}+\frac{1}{q}=1$ then $R(z):L^p(X)\to L^p(X)$ is bounded and by duality also $R(z):L^q(X)\to L^q(X)$ if and only if $a>-c_p$. 
\end{folg} 

\begin{lemma}
Let $z\in \C$. Then the Fourier transform of the resolvent $R(z)$, is given by 
\begin{align*}
\widetilde{R(z)f}^{\sigma}=-\frac{1}{(\lambda^2+\rho^2)+z}\tilde{f}^{\sigma} \quad\forall f\in C^{\infty}_c(X).
\end{align*}
\end{lemma}

\begin{proof}
Let $f\in C^{\infty}_c(X)$ then since the heat semi-group acts by convolution with the heat kernel on $C^{\infty}_c(X)$ and because of Lemma \ref{f+g} we have:
\begin{align*}
-\widetilde{R(z)f}^{\sigma}(\lambda,\xi)&=\int_0^{\infty}e^{-zt}\widetilde{e^{t\Delta}f}^{\sigma}(\lambda,\xi)\,dt\\
&=\int_0^{\infty}e^{-zt}e^{-(\lambda^2+\rho^2)t}\tilde{f}^{\sigma}(\lambda,\xi)\,dt\\
&=\frac{1}{(\lambda^2+\rho^2)+z}\tilde{f}^{\sigma}(\lambda,\xi).
\end{align*}
\end{proof}

\begin{prop}
Let $X$ be a non-compact simply connected harmonic manifold of rank one, $\frac{1}{p}+\frac{1}{q}=1$ and $p\leq2$. If the resolvent $R(z):L^p(X)\to L^p(X)$ is bounded for $z=\sigma+i\tau$ then
\begin{align*}
\tau>-4a^2\Bigl(\sigma+\frac{4\rho^2}{pq}\Bigr)
\end{align*}
where $a=\rho\Bigl(\frac{2}{p}-1\Bigr)$.
\end{prop}

\begin{proof}
Assume the resolvent is bounded then by Lemma \ref{lemma:holo2} for every $f\in L^p(X)$, $x\in X$ every $\lambda \in S_q$ and almost every $\xi\in \partial X$
\begin{align*}
\widetilde{R(z)f}^{x}(\lambda,\xi)=\frac{-1}{\lambda^2+\rho^2+z}f(\lambda,\xi)
\end{align*}
is holomorphic on $S_q$ and bounded on every closed sup strip. 
Hence poles have to lie outside every sub-strip. Let us now calculate these poles. 
Set $\lambda=u+iv$ then $\lambda^2+\rho^2+z$ yields:
\begin{align*}
u^2-v^2&=-\sigma-\rho^2\\
2uv&=\tau.
\end{align*}
We may assume $v\neq 0$ since otherwise $\tau=0$ and we would be in the case of Corollary \ref{folg;resolv1}.
Substitute $k=\sigma+\rho$ then with $u=\frac{\tau}{2v}$ and since $v^2>0$ we have:
\begin{align*}
v^2=\frac{k+\sqrt{k^2+\tau^2}}{2}.
\end{align*}
Because of the considerations above we have $\lvert v\rvert\geq a$ hence:
\begin{align*}
\sqrt{k^2+\tau^2}\geq 2a^2-k.
\end{align*}
Now the right-hand side is non-negative if and only if $\sigma\leq \rho^2(1-\frac{8}{pq})$. Therefore in this case:
\begin{align*}
\tau^2&\geq (2a^2-k)^2-k\\
&=-4a^2\Bigl(\sigma+\frac{4\rho^2}{pq}\Bigr).
\end{align*}
And the case $\sigma >\rho^2(1-\frac{8}{pq})$ implies the same. 
Now the resolvent set is open hence $"\geq"$ becomes $">"$. 
\end{proof}
\begin{proof}[Proof Theorem \ref{thm:reschaos}]
  In contrast to the Laplace operator the resolvent is, at least off the spectrum, a bounded linear operator and

has the symbol $-((\lambda^2+\rho^2)+z)^{-1}$. Furthermore by Lemma \ref{lemma:heatintegral} we observe that the resolvent of the Laplacian is an $L^p$ multiplier since the heat kernel is. 
 Now for $z$ close to the spectrum of $\Delta$ on $L^p(X)$ for $2<p<\infty$, discussed above,  we can find $\lambda_1,\lambda_2 \in S_p$ such that 
 \begin{align*}
 \lvert ((\lambda_1^2+\rho^2)+z)^{-1}\rvert \leq 1 \leq \lvert ((\lambda_2^2+\rho^2)+z)^{-1}\rvert.
 \end{align*}
Hence by the intermediate value theorem and Theorem \ref{chaotic1} $R(z):L^p(X)\to L^p(X)$ is chaotic for $2< p<\infty$.
\end{proof}

\subsection{Discussion}

In addition to some immediate corollaries regarding the chaotic resp. non-chaotic behaviour of $L^p$-multipliers and especially the heat semi-group on the space of radial functions we want to discuss the cases $p=\infty$ and $m_T$ not being a continuous function. 

First, let us consider the case $p=\infty$. It should be noted that  $L^{\infty}(X)$  is not a separable Banach space, but assume that the definition of chaos extends to $L^{\infty}(X)$ in a meaningful way. The condition $p<\infty$ first appears in Lemma \ref{lemma:dens}, the reason for this is that the proof uses a duality argument and the dual space of $L^{\infty}(X)$ is larger than $L^1(X)$. To see that our result is optimal 
 consider the bounded linear operator $T:L^{\infty}(X)\to L^{\infty}(X)$ given by the convolution with a radial $C^{\infty}_c$ function $f$ on $X$. Then $g*f$ is continuous hence $T$ maps $L^{\infty}(X)$ to the subspace of continuous functions. Therefore it can not be chaotic.\\
 The assumption on the symbol of the $L^p$-multiplier to be a holomorphic function in the case $p\leq 2$ is natural since the same holds true for $p>2$. 
Now consider $p=2$ and $T:L^2(X)\to L^2(X)$ such that
 \begin{align*}
 m_T(\lambda)=\begin{cases}
 -1&\text{if }\lambda\in (-1,1),\\
 1&\text{else.}
 \end{cases}
 \end{align*}
 One can easily see using the Plancherel theorem and the inverse Fourier transform that such an operator exists. Then
  again by the Plancherel theorem for $n\in\N$ we have $T^{2n}=\operatorname{id}$. Hence every function is a periodic point. \\


\newpage
\footnotesize

\bibliography{literature}
\bibliographystyle{alpha}

\end{document}